\documentclass[12pt]{amsart}

\usepackage{amsmath}
\usepackage{amssymb}
\usepackage{amsfonts}
\usepackage{amsthm}
\usepackage{enumerate}
\usepackage{hyperref}
\usepackage{color}

\theoremstyle{plain}
\newtheorem{thm}{Theorem}[section]

\newtheorem{prop}[thm]{Proposition}
\newtheorem{lem}[thm]{Lemma}
\newtheorem{cor}[thm]{Corollary}

\newtheorem{ques}[thm]{Question}

\newtheorem{conj}[thm]{Conjecture}

\theoremstyle{definition}

\newtheorem{dfns-rems}[thm]{Definitions and Remarks}
\newtheorem{notas-rems}[thm]{Notations and Remarks}
\newtheorem{exmps-rems}[thm]{Examples and Remarks}

\begin{document}


\title[sdepth of powers of edge ideals]{On the Stanley depth of powers of edge ideals}


\author[S. A. Seyed Fakhari]{S. A. Seyed Fakhari}

\address{S. A. Seyed Fakhari, School of Mathematics, Statistics and Computer Science,
College of Science, University of Tehran, Tehran, Iran.}

\email{fakhari@khayam.ut.ac.ir}

\urladdr{http://math.ipm.ac.ir/$\sim$fakhari/}


\begin{abstract}
Let $\mathbb{K}$ be a field and $S=\mathbb{K}[x_1,\dots,x_n]$ be the
polynomial ring in $n$ variables over $\mathbb{K}$. Let $G$ be a
graph with $n$ vertices. Assume that $I=I(G)$ is the edge ideal of $G$ and $p$ is the number of its bipartite connected components. We prove that for every positive integer $k$, the inequalities ${\rm sdepth}(I^k/I^{k+1})\geq p$ and ${\rm sdepth}(S/I^k)\geq p$ hold. As a consequence, we conclude that $S/I^k$ satisfies the Stanley's inequality for every integer $k\geq n-1$. Also, it follows that $I^k/I^{k+1}$ satisfies the Stanley's inequality for every integer $k\gg 0$. Furthermore, we prove that if (i) $G$ is a non-bipartite graph, or
(ii) at least one of the connected components of $G$ is a tree with at least one edge, then $I^k$ satisfies the Stanley's inequality for every integer $k\geq n-1$.  Moreover, we verify a conjecture of the author in special cases.
\end{abstract}






\thanks{}


\maketitle


\section{Introduction} \label{sec1}

Let $\mathbb{K}$ be a field and let $S=\mathbb{K}[x_1,\dots,x_n]$
be the polynomial ring in $n$ variables over $\mathbb{K}$. Let
$M$ be a finitely generated $\mathbb{Z}^n$-graded $S$-module. Let
$u\in M$ be a homogeneous element and $Z\subseteq
\{x_1,\dots,x_n\}$. The $\mathbb {K}$-subspace $u\mathbb{K}[Z]$
generated by all elements $uv$ with $v\in \mathbb{K}[Z]$ is
called a {\it Stanley space} of dimension $|Z|$, if it is a free
$\mathbb{K}[Z]$-module. Here, as usual, $|Z|$ denotes the number
of elements of $Z$. A decomposition $\mathcal{D}$ of $M$ as a
finite direct sum of Stanley spaces is called a {\it Stanley
decomposition} of $M$. The minimum dimension of a Stanley space
in $\mathcal{D}$ is called the {\it Stanley depth} of
$\mathcal{D}$ and is denoted by ${\rm sdepth} (\mathcal {D})$.
The quantity $${\rm sdepth}(M):=\max\big\{{\rm sdepth}
(\mathcal{D})\mid \mathcal{D}\ {\rm is\ a\ Stanley\
decomposition\ of}\ M\big\}$$ is called the {\it Stanley depth}
of $M$. We say that a $\mathbb{Z}^n$-graded $S$-module $M$ satisfies {\it Stanley's inequality} if $${\rm depth}(M) \leq
{\rm sdepth}(M).$$ In fact, Stanley \cite{s} conjectured that every $\mathbb{Z}^n$-graded $S$-module satisfies Stanley's inequality.
This conjecture has been recently disproved in \cite{abcj}.
However, it is still interesting to find the classes of
$\mathbb{Z}^n$-graded $S$-modules which satisfy Stanley's inequality.
For a reader friendly introduction to Stanley depth, we refer to
\cite{psty} and for a nice survey on this topic, we refer to
\cite{h}.

Let $G$ be a graph with vertex set $V(G)=\big\{v_1, \ldots,
v_n\big\}$. The edge ideal $I(G)$ of $G$ is the ideal of $S$
generated by the squarefree  monomials  $x_ix_j$, where $\{v_i, v_j\}$ is an edge of $G$. In \cite{psy}, the authors proved that if $G$ is a forest (i.e., a graph with no cycle), then $S/I(G)^k$ satisfies the Stanley's inequality for every integer $K\gg 0$. Also, it was shown in \cite{asy} that $I(G)^k/I(G)^{k+1}$ satisfies the Stanley's inequality for every forest $G$ and every integer $K\gg 0$. The aim of this paper is to extend theses results to the whole class of graphs. In Theorem \ref{quo}, we prove that for every graph $G$, the inequality ${\rm sdepth}(S/I(G)^k)\geq p$ holds, where $p$ is the number of bipartite components of $G$. Combining this inequality with a recent result of Trung \cite{t}, we conclude that $S/I(G)^k$ satisfies the Stanley's inequality for every integer $k\geq n-1$ (see Corollary \ref{sconj1}). In Theorem \ref{twoquo}, we study the Stanley depth of $I(G)^k/I(G)^{k+1}$ and prove that ${\rm sdepth}(I(G)^k/I(G)^{k+1})\geq p$, for every integer $k\geq 0$.  Combining this inequality with a result of Herzog and Hibi \cite{hh''}, we deduce that $I(G)^k/I(G)^{k+1}$ satisfies the Stanley's inequality for large $k$ (see Corollary \ref{sconj2}).

In section \ref{sec2}, we investigate the Stanley depth of $I(G)^k$, for a positive integer $k$. In Theorem \ref{ideal}, we determine a lower bound for the Stanley depth of $I(G)^k$. In Corollaries \ref{nonb} and \ref{ctree}, we prove that if (i) $G$ is a non-bipartite graph, or
(ii) at least one of the connected components of $G$ is a tree (i.e., a connected forest) with at least one edge, then for every positive integer $k$, the Stanley depth of $I(G)^k$ is at least one more than the number of bipartite connected components of $G$. Then we conclude that for theses classes of graphs, the ideal $I(G)^k$ satisfies the Stanley's inequality, for every $k\geq n-1$, where $n=\mid V(G)\mid$ (see Corollary \ref{sin}).


\section{Stanley depth of quotient of powers of edge ideals} \label{sec1}

In this section, we study the Stanley depth of quotient of powers of edge ideals. Before starting the proofs, we remind that for every graph $G$ and every subset $W$ of $V(G)$, the graph $G\setminus W$ is the graph formed by removing
the vertices of $W$ from the vertex set of $G$ and deleting any edge in $G$ that
contains a vertex of $W$.

The first main result of this paper asserts that the number of bipartite connected components of $G$ is a lower bound for the Stanley depth of $I(G)^k/I(G)^{k+1}$, for every nonnegative integer $k$. We first need the following lemma.

\begin{lem} \label{one}
Let $G$ be a connected bipartite graph with edge ideal $I=I(G)$. Then for every integer $k\geq 0$, we have ${\rm sdepth}(I^k/I^{k+1})\geq 1$.
\end{lem}

\begin{proof}
By \cite[Proposition 2.13]{br}, it is enough to prove that for every integer $k\geq 0$, we have ${\rm depth}(I^k/I^{k+1})\geq 1$. We use induction on $k$. For $k=0$, the assertion says that ${\rm depth}(S/I)\geq 1$ which is trivial. Thus, assume that $k\geq 1$. Consider the following short exact sequence:$$0\longrightarrow I^k/I^{k+1}\longrightarrow S/I^{k+1}\longrightarrow S/I^k \longrightarrow 0.$$It follows from depth lemma that$${\rm depth}(I^k/I^{k+1})\geq \min\{{\rm depth}(S/I^{k+1}), {\rm depth}(S/I^k)+1\}\geq 1,$$where the last inequality follows from \cite[Lemma 2.6]{m}.
\end{proof}

We are now ready to prove the first main result of this paper.

\begin{thm} \label{twoquo}
Let $G$ be a graph with edge ideal $I=I(G)$. Suppose that $p$ is the number of bipartite connected components of $G$. Then for every integer $k\geq 0$, we have ${\rm sdepth}(I^k/I^{k+1})\geq p$.
\end{thm}

\begin{proof}
Using \cite[Lemma 3.6]{hvz}, we may assume that $G$ has no isolated vertex. We prove the theorem by induction on $p$. There is nothing to prove if $p=0$. Thus, assume that $p\geq 1$ and the assertion is true for every graph with at most $p-1$ bipartite connected components. If $G$ is connected, it follows from $p\geq 1$ that $G$ is a bipartite graph and the claim follows from Lemma \ref{one}. Therefore, assume that $G$ has at least two connected components. Suppose that $G_1$ is a bipartite connected component of $G$. Let $L$ and $J$ be the edge ideals of $G_1$ and $G\setminus V(G_1)$, respectively (we consider $L$ and $J$ as ideals in $S$). Then $I=L+J$. Therefore, $I^k= \sum_{s+t=k}L^sJ^t$ and thus,
\[
\begin{array}{rl}
I^k/I^{k+1}=\sum_{s+t=k}(L^sJ^t+I^{k+1}/I^{k+1}).
\end{array} \tag{1} \label{1}
\]
On the other hand, for every distinct pairs $(s, t)\neq (l,m)$ of nonnegative integers with $s+t=l+m=k$, we have
\begin{align*}
L^sJ^t\cap L^lJ^m & = (L^s\cap J^t)\cap (L^l\cap J^m)=L^{\max\{s,l\}} \cap J^{\max\{t,m\}}\\ & =
L^{\max\{s,l\}} J^{\max\{t,m\}}\subseteq (L+J)^{k+1}=I^{k+1}.
\end{align*}
This shows that the sum in (\ref{1}) is direct and therefore by the definition of Stanley depth we have$${\rm sdepth}(I^k/I^{k+1})\geq \min_{s+t=k} \{{\rm sdepth}(L^sJ^t+I^{k+1}/I^{k+1})\}.$$Hence, it is enough to show that for every pair $(s, t)$ of nonnegative integers with $s+t=k$,$${\rm sdepth}(L^sJ^t+I^{k+1}/I^{k+1})\geq p,$$or equivalently$${\rm sdepth}(L^sJ^t/L^sJ^t\cap I^{k+1})\geq p.$$

Note that for every pair $(s, t)$ of nonnegative integers with $s+t=k$,
\begin{align*}
& L^{s+1}J^t+L^sJ^{t+1} \subseteq L^sJ^t\cap I^{k+1}=L^sJ^t\cap (L+J)^{k+1}\subseteq L^s\cap J^t\cap (L^{s+1}+ J^{t+1})=\\ & (L^s\cap J^t\cap L^{s+1})+ (L^s\cap J^t\cap J^{t+1})= (J^t\cap L^{s+1})+ (L^s\cap J^{t+1})\\ & = L^{s+1}J^t+ L^sJ^{t+1}.
\end{align*}
This shows that$$L^{s+1}J^t+L^sJ^{t+1}=L^sJ^t\cap I^{k+1}$$and thus,$$L^sJ^t/L^sJ^t\cap I^{k+1}=L^sJ^t/(L^{s+1}J^t+L^sJ^{t+1}).$$Set $S'=\mathbb{K}[x_i\mid v_i\in V(G_1)]$ and $S''=\mathbb{K}[x_i\mid v_i\notin V(G_1)]$. Also, set $L'=L\cap S'$ and $J'=J\cap S''$. By Lemma \ref{one} ${\rm sdepth}_{S'}(L'^s/L'^{s+1})\geq 1$ and by induction hypothesis ${\rm sdepth}_{S''}(J'^t/J'^{t+1})\geq p-1$. Hence, there exist Stanley decompositions $$\mathcal{D} \ : \ L'^s/L'^{s+1}=\bigoplus_{i=1}^ru_i\mathbb{K}[Z_i] \ \ \ \ \ {\rm and} \ \ \ \ \ \mathcal{D}' \ : \ J'^t/J'^{t+1}=\bigoplus_{j=1}^{r'}u'_j\mathbb{K}[Z'_j]$$with ${\rm sdepth}(\mathcal{D})\geq 1$ and ${\rm sdepth}(\mathcal{D}')\geq p-1$. One can easily check that $$L^sJ^t/(L^{s+1}J^t+L^sJ^{t+1})=\bigoplus_{i=1}^r\bigoplus_{j=1}^{r'}u_iu'_j\mathbb{K}[Z_i\cup Z'_j]$$ is a Stanley decomposition and since for every pair of integers $i$ and $j$ with $1\leq i\leq r$ and $1\leq j\leq r'$, we have $Z_i\cap Z'_j=\emptyset$, it follows that$${\rm sdepth}(L^sJ^t/L^sJ^t\cap I^{k+1})={\rm sdepth}(L^sJ^t/(L^{s+1}J^t+L^sJ^{t+1}))\geq p$$and this completes the proof.
\end{proof}

The following theorem is a consequence of Theorem \ref{twoquo}. It shows that the number of bipartite connected components of $G$ is a lower bound for the Stanley depth of $S/I(G)^k$, for every positive integer $k$.

\begin{thm} \label{quo}
Let $G$ be a graph with edge ideal $I=I(G)$. Suppose that $p$ is the number of bipartite connected components of $G$. Then for every integer $k\geq 1$, we have ${\rm sdepth}(S/I^k)\geq p$.
\end{thm}

\begin{proof}
We use induction on $k$. For $k=1$, it follows from Theorem \ref{twoquo} that ${\rm sdepth}(S/I)={\rm sdepth}(I^0/I)\geq p$. Thus, assume that $k\geq 2$ and ${\rm sdepth}(S/I^{k-1})\geq p$. Consider the following short exact sequence:$$0\longrightarrow I^{k-1}/I^k\longrightarrow S/I^k\longrightarrow S/I^{k-1} \longrightarrow 0.$$Using \cite[Lemma 2.2]{r'}, we conclude that$${\rm sdepth}(S/I^k)\geq \min\{{\rm sdepth}(I^{k-1}/I^k), {\rm sdepth}(S/I^{k-1})\}.$$Now, Theorem \ref{twoquo} and the induction hypothesis imply that ${\rm sdepth}(S/I^k)\geq p$.
\end{proof}

Let $\mathbb{K}$ be a field and $S=\mathbb{K}[x_1,x_2,\dots,x_n]$ be the
polynomial ring in $n$ variables over the field $\mathbb{K}$, and let
$I\subset S$ be a monomial ideal. A classical result by Burch \cite{b'} states
that $$\min_k{\rm depth}(S/I^k)\leq n-\ell(I),$$ where $\ell(I)$ is the
analytic spread of $I$, that is, the dimension of $\mathcal{R}(I)/
{{\frak{m}}\mathcal{R}(I)}$, where $\mathcal{R}(I)=\bigoplus_
{n=0}^{\infty}I^n= S[It] \subseteq S[t]$ is the Rees ring of $I$ and
$\frak{m}=(x_1,\ldots,x_n)$ is the maximal ideal of $S$. By a theorem of
Brodmann \cite{b}, ${\rm depth}(S/I^k)$ is constant for large $k$. We call
this constant value the {\it limit depth} of $I$, and denote it by
$\lim_{t\rightarrow \infty}{\rm depth}(S/I^k)$. Brodmann improved the Burch's
inequality by showing that$$\lim_{k\rightarrow \infty}{\rm depth}(S/I^k)
\leq n-\ell(I).$$

Let $G$ be a graph with edge ideal $I=I(G)$. Suppose that $n$ is the number of vertices of $G$ and $p$ is the number of its bipartite connected components. It follows from \cite[Page 50]{v} that $\ell(I)=n-p$. Thus, using the Burch's inequality, we conclude that $$\lim_{k\rightarrow \infty}{\rm depth}(S/I^k)
\leq p.$$Recently, Trung \cite{t} proved that we have in fact equality in the above inequality. Indeed, he proved the following stronger result.

\begin{thm} \label{trung}
{\rm (} \cite[Theorems 4.4 and 4.6]{t} {\rm )} Let $G$ be a graph with edge ideal $I=I(G)$. Suppose that $n$ is the number of vertices of $G$ and $p$ is the number of its bipartite connected components. Then for every integer $k\geq n-1$, we have ${\rm depth}(S/I^k)=p$.
\end{thm}

In \cite[Corollary 2.8]{psy}, the authors proved that $S/I^k$ satisfies the Stanley's inequality for every $k\gg0$, when $I$ is the edge ideal of a forest. The following corollary is an extension of this result and it is an immediate consequence of Theorems \ref{quo} and \ref{trung}.

\begin{cor} \label{sconj1}
Let $G$ be a graph with edge ideal $I=I(G)$. Suppose that $n$ is the number of vertices of $G$. Then $S/I^k$ satisfy the Stanley's inequality, for every integer $k\geq n-1$.
\end{cor}

In \cite[Corollary 3.2]{asy}, the authors proved that $I^k/I^{k+1}$ satisfies the Stanley's inequality for every $k\gg0$, when $I$ is the edge ideal of a forest. The following corollary is an extension of this result and shows that $I^k/I^{k+1}$ satisfies the Stanley's inequality for every edge ideal $I$ and every integer $k\gg0$. Unfortunately, we are not able to determine an upper bound for the least integer $k$, such that $I^k/I^{k+1}$ satisfies the Stanley's inequality.

\begin{cor} \label{sconj2}
Let $G$ be a graph with edge ideal $I=I(G)$. Then $I^k/I^{k+1}$ satisfies the Stanley's inequality, for every integer $k\gg 0$.
\end{cor}

\begin{proof}
It follows from \cite[Theorem 1.2]{hh''} and \cite[Theorem 4.4]{t} that
$$\lim_{k\rightarrow\infty}{\rm depth}(I^k/I^{k+1})= \lim_{k\rightarrow\infty}{\rm depth}(S/I^k)=p.$$Now, Theorem \ref{twoquo} completes the proof.
\end{proof}

Let $I\subset S$ be an arbitrary ideal. An element $f \in S$ is
{\it integral} over $I$, if there exists an equation
$$f^k + c_1f^{k-1}+ \ldots + c_{k-1}f + c_k = 0 {\rm \ \ \ \ with} \ c_i\in I^i.$$
The set of elements $\overline{I}$ in $S$ which are integral over $I$ is the {\it integral closure}
of $I$. The ideal $I$ is {\it integrally closed}, if $I = \overline{I}$.

In \cite{s1}, the author proposed the following conjecture regarding the Stanley depth of integrally closed monomial ideals.

\begin{conj} \label{conje}
{\rm (}\cite[Conjecture 2.6]{s1}{\rm )} Let $I\subset S$ be an integrally closed monomial ideal. Then ${\rm sdepth}(S/I)\geq n-\ell(I)$ and ${\rm sdepth} (I)\geq n-\ell(I)+1$.
\end{conj}

This conjecture is known to be true for some classes of monomial ideals. For example, it is shown in \cite[Theorem 2.5]{psy1} that the  conjecture is true for every weakly polymatroidal ideal which is generated in a single degree. Also, in \cite[Corollary 3.4]{s2}, the author proved the conjecture for every squarefree monomial ideal which is generated in a single degree. Now, Theorem \ref{quo} shows that the conjectured inequality for $S/I$ is true, when $I$ is a power of the edge ideal of a graph.


\section{Stanley depth of powers of edge ideals} \label{sec2}

In this section we determine a lower bound for the Stanley depth of powers of edge ideals. In particular, we show that if either $G$ is a non-bipartite graph, or has a connected component which is a tree with at least one edge, then for every positive integer $k$, the Stanley depth of $I(G)^k$ is at least one more than the number of the bipartite connected components of $G$. We conclude that in theses cases $I(G)^k$ satisfies the Stanley's inequality for every $k\geq \mid V(G)\mid-1$.

\begin{thm} \label{ideal}
Let $G$ be a graph with edge ideal $I=I(G)$. Assume that $H$ is a connected component of $G$ with at least one edge. Suppose that $h$ is the number of bipartite connected components of $G\setminus V(H)$. Then for every integer $k\geq 1$, we have$${\rm sdepth}(I^k)\geq \min_{1\leq l \leq k}\{{\rm sdepth}_{S'}(I(H)^l)\}+h,$$where $S'=\mathbb{K}[x_i\mid v_i\in V(H)]$.
\end{thm}

\begin{proof}
Using \cite[Lemma 3.6]{hvz}, we may assume that $G\setminus V(H)$ has no isolated vertex. We prove the theorem by induction on the number of connected components of $G$. If $G$ is connected, the $G=H$ and there is nothing to prove. Thus, assume that $G$ is not a connected graph. Set $S''=\mathbb{K}[x_i\mid v_i\notin V(H)]$. Let $L'\subset S'$ and $J' \subset S''$ be the edge ideals of $H$ and $G\setminus V(H)$, respectively. Set $L=L'S$ and $J=J'S$. Then $I=L+J$. Therefore, $I^k= \sum_{t=0}^kL^tJ^{k-t}$. For every integer $l$ with $1\leq l\leq k$, we have the following short exact sequence:$$0\longrightarrow \sum_{t=0}^{l-1}L^tJ^{k-t}\longrightarrow \sum_{t=0}^lL^tJ^{k-t}\longrightarrow \sum_{t=0}^lL^tJ^{k-t}/\sum_{t=0}^{l-1}L^tJ^{k-t} \longrightarrow 0.$$Using \cite[Lemma 2.2]{r'}, we conclude that$${\rm sdepth}(\sum_{t=0}^lL^tJ^{k-t})\geq \min\{{\rm sdepth}(\sum_{t=0}^{l-1}L^tJ^{k-t}), {\rm sdepth}(\sum_{t=0}^lL^tJ^{k-t}/\sum_{t=0}^{l-1}L^tJ^{k-t})\}.$$This implies that
\begin{align*}
{\rm sdepth}(I^k) & ={\rm sdepth}(\sum_{t=0}^kL^tJ^{k-t})\\ & \geq\min\{{\rm sdepth}(J^k), {\rm sdepth}(\sum_{t=0}^lL^tJ^{k-t}/\sum_{t=0}^{l-1}L^tJ^{k-t}): l=1, \ldots, k\}.
\end{align*}
Let $H'$ be a connected component of $G\setminus V(H)$ and set $T=\mathbb{K}[x_i\mid v_i\in V(H')]$. Notice that $G\setminus (V(H)\cup V(H'))$ has at least $h-1$ bipartite connected components. Also, note that $H'$ has at leat one edge and by \cite[Corollary 24]{h}, we have ${\rm sdepth}_{T}(I(H')^l)\geq 1$, for every positive integer $l$. Therefore, using induction hypothesis and \cite[Lemma 3.6]{hvz}, we conclude that
\begin{align*}
& {\rm sdepth}(J^k)={\rm sdepth}_{S''}(J'^k)+\mid V(H)\mid\geq \min_{1\leq l \leq k}\{{\rm sdepth}_{T}(I(H')^l)\}+(h-1)+\mid V(H)\mid\\ & \geq 1+(h-1)+\min_{1\leq l \leq k}\{{\rm sdepth}_{S'}(I(H)^l)\}=\min_{1\leq l \leq k}\{{\rm sdepth}_{S'}(I(H)^l)\}+h.
\end{align*}
Therefore, it is enough to show that for every integer $l$ with $1\leq l\leq k$, the inequality$${\rm sdepth}(\sum_{t=0}^lL^tJ^{k-t}/\sum_{t=0}^{l-1}L^tJ^{k-t})\geq \min_{1\leq l \leq k}\{{\rm sdepth}_{S'}(I(H)^l)\}+h$$holds. Notice that $$\sum_{t=0}^lL^tJ^{k-t}/\sum_{t=0}^{l-1}L^tJ^{k-t}\cong L^lJ^{k-l}/(L^lJ^{k-l}\cap \sum_{t=0}^{l-1}L^tJ^{k-t})$$and
\begin{align*}
& L^lJ^{k-l}\cap \sum_{t=0}^{l-1}L^tJ^{k-t}=\sum_{t=0}^{l-1}(L^lJ^{k-l}\cap L^tJ^{k-t})=\sum_{t=0}^{l-1}(L^l\cap J^{k-l}\cap L^t\cap J^{k-t})=\\ & \sum_{t=0}^{l-1}(L^l\cap J^{k-t})=\sum_{t=0}^{l-1}L^lJ^{k-t}=L^l\sum_{t=0}^{l-1}J^{k-t}=L^lJ^{k-l+1}.
\end{align*}
Hence,
\[
\begin{array}{rl}
\sum_{t=0}^lL^tJ^{k-t}/\sum_{t=0}^{l-1}L^tJ^{k-t}\cong L^lJ^{k-l}/L^lJ^{k-l+1}.
\end{array} \tag{2} \label{2}
\]
By Theorem \ref{twoquo} ${\rm sdepth}_{S''}(J'^{k-l}/J'^{k-l+1})\geq h$. Consider the Stanley decompositions $$\mathcal{D} \ : \ L'^l=\bigoplus_{i=1}^ru_i\mathbb{K}[Z_i] \ \ \ \ \ {\rm and} \ \ \ \ \ \mathcal{D}' \ : \ J'^{k-l}/J'^{k-l+1}=\bigoplus_{j=1}^{r'}u'_j\mathbb{K}[Z'_j]$$with ${\rm sdepth}(\mathcal{D})={\rm sdepth}(L'^l)$ and ${\rm sdepth}(\mathcal{D}')\geq h$. One can easily check that $$L^lJ^{k-l}/L^lJ^{k-l+1})=\bigoplus_{i=1}^r\bigoplus_{j=1}^{r'}u_iu'_j\mathbb{K}[Z_i\cup Z'_j]$$ is a Stanley decomposition and since for every pair of integers $i$ and $j$ with $1\leq i\leq r$ and $1\leq j\leq r'$, we have $Z_i\cap Z'_j=\emptyset$, it follows that$${\rm sdepth}(L^lJ^{k-l}/L^lJ^{k-l+1})\geq {\rm sdepth}(L'^l)+h\geq \min_{1\leq l \leq k}\{{\rm sdepth}_{S'}(I(H)^l)\}+h$$and thus, the isomorphism (\ref{2}) completes the proof.
\end{proof}

The following corollary shows that if $G$ has a non-bipartite connected component, then for every positive integer $k$, the Stanley depth of $I(G)^k$ is at least one more than the number of bipartite connected components of $G$.

\begin{cor} \label{nonb}
Let $G$ be a non-bipartite graph with edge ideal $I=I(G)$. Suppose that $p$ is the number of bipartite connected components of $G$. Then for every integer $k\geq 1$, we have ${\rm sdepth}(I^k)\geq p+1$.
\end{cor}

\begin{proof} 
Note that $G$ has a non-bipartite connected component, say $H$. Thus, the assertion follows by applying Theorem \ref{ideal} and \cite[Corollary 24]{h}.
\end{proof}

In view of Corollary \ref{nonb}, we expect that the inequality ${\rm sdepth}(I(G)^k)\geq p+1$ can be true for every graph $G$ with $p$ bipartite connected components and for every positive integer $k$. In order to prove this inequality one only needs to prove it when $G$ is a connected bipartite graph (with at least one edge). Then the desired inequality follows from Theorem \ref{ideal}. Thus, one can ask the following question.

\begin{ques} \label{quest}
Let $G$ be a connected bipartite graph (with at least one edge) and suppose $k\geq 1$ is an integer. Is it true that ${\rm sdepth}(I(G)^k)\geq 2$?
\end{ques}

By \cite[Corollary 3.4]{s2} and \cite[Page 50]{v}, we know that the answer of Question \ref{quest} is positive for $k=1$. Unfortunately, we are not able to give a complete answer to Question \ref{quest}. However, we give a positive answer to this question, when $G$ is a tree. We first need to
introduce some basic notions from graph theory.

Let $G$ be a graph with vertex set $V(G)=\big\{v_1, \ldots,
v_n\big\}$. For a vertex $v_i$, the {\it neighbor set} of $v_i$ is
$N(v_i)=\big\{v_j \mid \{v_i, v_j\} \ {\rm is\ an\ edge\ of\ } G\big\}$.
The vertex $v_i$ is called a {\it leaf}  if $N(v_i)$ is a
singleton. 

\begin{prop} \label{tree}
Let $G$ be a tree with at least one edge. Then for every integer $k\geq 1$, we have ${\rm sdepth}(I(G)^k)\geq 2$.
\end{prop}

\begin{proof}
Set $I=I(G)$. We prove the claim by induction on $n+k$, where $n$ is the number
of  vertices of $G$. If $k=1$, then the result follows from \cite[Corollary 3.4]{s2} and \cite[Page 50]{v}.
 If $n=2$, then $I^k$ is a
principal ideal and the assertion is trivially true. We therefore suppose that $k\geq 2$
and $n\geq 3$. Let $v_1$ be a leaf of $G$ and assume that
$N(v_1)=\{v_2\}$. Let
$S'=\mathbb{K}[x_2, \ldots, x_n]$ be the polynomial ring
obtained from  $S$ by deleting the variable $x_1$. Then
$$I^k=(I^k\cap S')\bigoplus x_1(I^k:x_1)$$and therefore by definition
of the Stanley depth it is enough to prove that
\begin{itemize}
\item[(i)] ${\rm sdepth}_{S'}(I^k\cap S')\geq 2$.

\item[(ii)] ${\rm sdepth}_S(I^k:x_1)\geq 2$.
\end{itemize}

To prove (i), let $I'\subseteq S'$ be the edge ideal of $G\setminus \{v_1\}$.
Then $I^k\cap S'=I'^k$. Since
$G\backslash \{v_1\}$ is a tree with $n-1$ vertices, the induction
hypothesis implies that$${\rm sdepth}_{S'}(I^k\cap S')\geq 2.$$
Next, we show that ${\rm sdepth}_S(I^k:x_1)\geq 2$.

Let $S''=\mathbb{K}[x_1, x_3, \ldots, x_n]$ be
the polynomial ring obtained from  $S$ by deleting the variable
$x_2$. Since
$$(I^k:x_1)=((I^k:x_1)\cap S'')\bigoplus
x_2(I^k:x_1x_2),$$by \cite[Lemma 2.10]{m}, we conclude that
$$(I^k:x_1)=((I^k:x_1)\cap S'')\bigoplus x_2I^{k-1}.$$Using the induction hypothesis, it follows that ${\rm sdepth}(I^{k-1})\geq 2$. Thus, to complete the proof we should show that if $(I^k:x_1)\cap S''\neq 0$, then $${\rm
sdepth}_{S''}((I^k:x_1)\cap S'')\geq 2.$$

Thus assume that $(I^k:x_1)\cap S''\neq 0$. Set $G'=G\backslash\{x_1, x_2\}$. Since $N(x_1)=\{x_2\}$, it follows$$(I^k:x_1)\cap S''=I(G')^kS''.$$Set $T={\mathbb
K}[x_3, \ldots, x_n]$. Now, \cite[Lemma 3.6]{hvz} and \cite[Corollary 24]{h} imply that $${\rm
sdepth}_{S''}((I^k:x_1)\cap S'')={\rm
sdepth}_{S''}(I(G')^kS'')={\rm sdepth}_T(I(G')^k)+1\geq 2.$$
\end{proof}

The following Corollary is a consequence of Theorem \ref{ideal} and Proposition \ref{tree}.

\begin{cor} \label{ctree}
Let $G$ be a graph with edge ideal $I=I(G)$. Assume that $G$ has $p$ bipartite connected components and suppose that at least on of the connected components of $G$ is tree with at least one edge. Then for every integer $k\geq 1$, we have ${\rm sdepth}(I^k)\geq p+1$.
\end{cor}

\begin{proof}
Apply Theorem \ref{ideal} by assuming that $H$ is the connected component of $G$ which is a tree with at least one edge. Note that by Proposition \ref{tree}, for every integer $k\geq 1$,$$\min_{1\leq l \leq k}\{{\rm sdepth}_{S'}(I(H)^l)\geq 2,$$where $S'=\mathbb{K}[x_i\mid v_i\in V(H)]$.
\end{proof}

As an immediate consequence of Theorem \ref{trung}, Corollaries \ref{nonb} and \ref{ctree}, we obtain the following result.

\begin{cor} \label{sin}
Assume that $G$ is a graph with $n$ vertices, such that
\begin{itemize}
\item[(i)] $G$ is a non-bipartite graph, or
\item[(ii)] at least one of the connected components of $G$ is a tree with at least one edge.
\end{itemize}
Then for every integer $k\geq n-1$, the ideal $I(G)^k$ satisfies the Stanley's inequality.
\end{cor}

We remind that Conjecture \ref{conje} predicts that $n-\ell(I)+1$ is a lower bound for the Stanley depth of $I$, when $I$ is an integrally closed monomial ideal of $S$. Note that by Corollaries \ref{nonb} and \ref{ctree}, the conjectured inequality is true, when $I$ is a power of the edge ideal of those graphs which belong to the classes (i) and (ii) of the above corollary.




\end{document}